\documentclass[11pt,a4paper,reqno]{amsart}
\usepackage[latin1]{inputenc}
\usepackage[english]{babel}
\usepackage{amsmath}
\usepackage{amsfonts}
\usepackage{amssymb}
\usepackage{mathrsfs}
\usepackage{latexsym}
\usepackage{yfonts}
\usepackage{natbib}
\usepackage[margin=3cm]{geometry}

\usepackage{color}

\usepackage{graphicx,color}

\usepackage[plainpages=false,colorlinks,hyperindex,bookmarksopen,linkcolor=red,citecolor=blue,urlcolor=blue]{hyperref}
\newcommand*\bigcdot{\mathpalette\bigcdot@{.3}}

\bibpunct{[}{]}{;}{n}{,}{,}

\theoremstyle{definition}

\newtheorem{theorem}{Theorem}[section]

\newtheorem{corollary}{Corollary}[section]

\newtheorem{remark}{Remark}[section]

\numberwithin{equation}{section}

\begin{document}
	\title{Variance Gamma (non-local) equations}
	\author{Fausto Colantoni}

	\address{Department of Basic and Applied Sciences for Engineering\\
		Sapienza University of Rome, Rome Italy}
		\email{fausto.colantoni@uniroma1.it}
	\begin{abstract}
		We provide some equations for the Variance Gamma process due to the fact that we do not consider only the definition as a time-changed Brownian motion. This brings us to a new non-local equation, even true in the drifted case, involving generalized Weyl derivatives. Then we focus on the connection to special functions and we study a space equation for our process. At the end, we conclude by observing the convergence in distribution of a compound Poisson process to the Variance Gamma process.		
	\end{abstract}

	\maketitle
	
	\section{Introduction}
	The Variance Gamma process is a famous L$\acute{e}$vy process used in mathematical finance (see \cite{madan1998variance}), also known as Laplace motion (\cite{kotz2001laplace}). It can be obtained by considering a Brownian motion
	with a random time given by an independent Gamma subordinator. The concept of subordination has been introduced by Bochner (\cite{bochner1949diffusion}) and, as for the other subordinated processes, we can associate the Phillips' operator (\cite{phillips1952generation}) for the governing equation. Recently, some new equations for the Variance Gamma process and the Gamma subordinator are provided in \cite{beghin2014geometric}, involving time-operators differently from classic theory on space operators.  
	
	At the current stage the non-local equation for the Variance Gamma process, defined as difference of two independent Gamma subordinators, is not considered and also a deeper analysis concerning this process and the equations for the modified Bessel functions is possible. In order to close such a gap, we focus on these equations. Then, we continue our dissertation by examining the similar version of the non-local equation for the Variance Gamma process with a drift. This is possible because, also in presence of drift, the definition as difference of two independent Gamma subordinators holds. At the end, we consider the compound Poisson process, related to the Gamma subordinator (the construction for any subordinator is developed in \cite[Proposition 3.3]{toaldo}), and its convergence (in distribution) to the Variance Gamma process. 
	
	\section{Preliminaries}
	Let $\Phi: (0, \infty) \mapsto (0, \infty)$ be a Bernstein function, which is uniquely defined by the so-called Bernstein representation
	\begin{align*}
	\Phi(\lambda)=\int_0^\infty (1-e^{-\lambda z}) \Pi(dz) , \quad \lambda > 0
	\end{align*}
	where $\Pi$ on $(0,\infty)$ with $\int_0^\infty (1 \wedge z) \Pi(dz) <\infty$ is the associated L$\acute{e}$vy measure. We also recall that
	\begin{align}
	\label{lapTailLm}
	\frac{\Phi(\lambda)}{\lambda}=\int_0^\infty e^{-\lambda z} \overline{\Pi}(z) dz, \quad \lambda>0
	\end{align}
	where $\overline{\Pi}(z)=\Pi((z,\infty))$ is termed \textit{tail of the L$\acute{e}$vy measure}.

	We focus only on the Laplace symbol
	\begin{align}
	\label{symbGamma}
	\Phi(\lambda) = a \ln \left(1 + \frac{\lambda}{b} \right) = a \int_0^\infty \left(1 - e^{- \lambda y} \right) \frac{e^{- b y}}{y}\, dy, \quad \lambda > 0, \quad a>0, \; b>0.
	\end{align}
	Thus, in this case, the L$\acute{e}$vy measure is $\Pi_{a,b}(dy)=a\frac{e^{- b y}}{y}\, dy$ and the associated Gamma subordinator $H=\{H_t, t \geq 0\}$, starting from zero, is such that
	\begin{align}
	\label{LapH}
	\mathbf{E}_0[e^{-\lambda H_t}] = e^{-t \Phi(\lambda)}, \quad \lambda > 0
	\end{align}
	where $\mathbf{P}_x$ denotes the probability measure for a process started from $x$ at time $t=0$ and $\mathbf{E}_x$ the mean value with respect to $\mathbf{P}_x$.
	The interest reader can consult \cite[Chapter III]{bertoin1996levy} for more details on subordinators. Since $\Pi_{a,b}(0,\infty)=\infty$, then from \cite[Theorem 21.3]{ken1999levy}, we have that $H$ has increasing sample path with jumps. We use the notation
	\begin{align*}
	\mathbf{P}(H_t \in dx) = h(t,x)\,dx.
	\end{align*}
	In addition, it is well known that, $\forall\, t > 0$,
	\begin{align}
	\label{repHsing} 
	h(t,x) 
	= & \left\lbrace
	\begin{array}{ll}
	\displaystyle \frac{b^{at}}{\Gamma(at)} x^{at-1} e^{-bx}, & x>0\\
	\displaystyle  0, & x \leq 0
	\end{array}
	\right.
	\end{align}
	trivially verifies
	\begin{align}
	\label{laph}
	\displaystyle \int_0^\infty e^{-\lambda x} h(t,x)\, dx = \frac{b^{at}}{(\lambda + b)^{at}} = e^{- t\, a\,\ln\left(1+ \frac{\lambda}{b}\right)}, \quad \lambda>0, \ t >0,
	\end{align}
	which coincides with formula \eqref{LapH}. We observe that the continuity of the function $h(t,x)$, when $x\to 0$, depends by the time variable $t$, indeed the Gamma subordinator has the \textit{time dependent property} (see \cite[Chapter 23]{ken1999levy}). From this, also the Variance Gamma process inherits continuity problems for its probability density function.
	
	Let $B:=\{B_t, t \geq 0\}$ be the one dimensional Brownian motion starting from zero, independent from $H$, and $g(t, x) = e^{-x^2/ 4t}/\sqrt{4 \pi t}$ be its probability density function. The Variance Gamma process $X:=\{X_t, t \geq 0\}$ can be defined as $X=B_H := B \circ H$, then it is a Brownian motion time-changed with a random clock given by an independent Gamma subordinator. Its probability density function is
	\begin{align}
	\label{ghVG}
	p(t,x)=\int_0^\infty g(s,x) h(t,s) ds,
	\end{align} 
	and the L$\acute{e}$vy symbol is $\Phi(\xi^2)$, as we see from
	\begin{align}
	\label{charVG}
	\mathbf{E}_0[e^{i \xi X_t}]=\hat{p}(t,x)=\int_{-\infty}^\infty e^{i \xi x} p(t,x) dx &= \int_{-\infty}^\infty e^{i \xi x} \int_0^\infty g(s,x) h(t,s) ds \, dx \notag\\ 
	&=\int_0^\infty e^{-s \xi^2} h(t,s) ds \notag\\ 
	&=e^{-t \Phi(\xi^2)}=e^{-at \ln\left(1+\frac{\xi^2}{b} \right)}.
	\end{align}
	For the Variance Gamma process we know an explicit representation for $p(t,x)$. From [\cite{table}, formula 3.478] we have
	\begin{align}
	\label{4.28a}
	\displaystyle
	\int_{0}^{\infty} x^{\nu -1} \exp \{-\beta x^{q} -\alpha x^{-q} \} dx= \frac{2}{q} \left(\frac{\alpha}{\beta}\right)^{\frac{\nu}{2q}} K_{\frac{\nu}{q}}\left(2 \sqrt{\alpha \beta}\right) \ , \ \ \ q, \alpha, \beta , \nu >0
	\end{align}
	where $K_\nu$ is the modified Bessel function, so we get that
	\begin{align}
	\label{gh}
	\int_0^\infty g(s,x) \, h
	(t,s)\, ds &= \int_0^\infty \frac{e^{-\frac{x^2}{4s}}}{\sqrt{4\pi s}} \frac{b^{at}}{\Gamma(at)} s^{at-1}  e^{-b s}\, ds= \notag
	\\
	&= \frac{b^{at}}{\sqrt{4 \pi}} \frac{1}{\Gamma(at)} \int_{0}^{\infty} \frac{s^{at-1}}{\sqrt{s}} \exp\left(-\frac{x^2}{4}s^{-1} -bs \right) ds,
	\end{align}
	then we use (\ref{4.28a}) by choosing $\nu=at-\frac{1}{2}, \ q=1, \ \alpha=\frac{z^2}{4}, \ \beta=b$ and we obtain
	\begin{align}
	\label{densityVG}
	p(t,x)=\int_0^\infty g(s,x) \, h
	(t,s)\, ds &=\frac{b^{at}}{\sqrt{4 \pi}} \frac{1}{\Gamma(at)} \ 2 \left(\frac{x^2}{4b}\right)^{\frac{1}{2} (at -\frac{1}{2})} K_{at-\frac{1}{2}} \left(2 \sqrt{\frac{x^2}{4} b} \right) = \notag \\
	&=\frac{b^{at}}{\sqrt{4 \pi}} \frac{1}{\Gamma(at)} \ 2 \left(\frac{x^2}{4b}\right)^{\frac{1}{2} (at -\frac{1}{2})} K_{at-\frac{1}{2}} \left(\vert x \vert \sqrt {b} \right)= \notag\\
	&=\frac{b^{at}}{\sqrt{\pi}} \frac{1}{\Gamma(at)} \left(\frac{1}{2 \sqrt{b}}\right)^{(at -\frac{1}{2})} \vert x \vert ^{at-\frac{1}{2}} \ K_{at-\frac{1}{2}}(\vert x \vert \sqrt{b}).
	\end{align}
	
	Since we are dealing with a time-changed stochastic process, we can use the Phillips' representation (or Bochner's subordination) to generate a new operator through subordination. It is well known that the Phillips' operator $-\Phi(-\Delta)$ (\cite{phillips1952generation}) can be written as
	\begin{align}
	\label{Phillips}
	-\Phi(-\Delta) u(x)= a \int_{0}^\infty (G_y u(x) - u(x) ) \frac{e^{- b y}}{y} dy,  
	\end{align}
	where $\Phi$ is defined in \eqref{symbGamma} and $G$ is the semigroup associated to the Brownian motion on $\mathbb{R}$, such that 
	the characteristic symbol is $\widehat{G}_t=e^{-t\xi^2}$. This leads us to
	\begin{align}
	\label{charPhillips}
	\int_{-\infty}^\infty e^{i \xi x} \left(-\Phi(-\Delta) u(x) \right) dx&=\left(a\int_{0}^\infty (e^{-y\xi^2}-1)  \frac{e^{- b y}}{y} dy\right) \hat{u}(\xi) \notag \\
	&= \left({-a \ln\left(1+\frac{\xi^2}{b} \right)} \right) \hat{u}(\xi).
	\end{align}
	If we combine \eqref{charVG} and \eqref{charPhillips}, we trivially obtain the well known result
	\begin{align}
	\label{eq:Phillips}
	\frac{\partial}{\partial t} p(t,x)= -\Phi(-\Delta) p(t,x), \quad t>0, \, x\in \mathbb{R}
	\end{align} 
	with $p(t,0)=\delta(x)$.
	
	If we change prospective and consider the operator in time, an interesting equation, presented in \cite[Remark 3.3]{beghin2014geometric}, for the Variance Gamma process is
	\begin{align*}
	\frac{\partial^2}{\partial x^2} p(t,x)=b \left(p(t,x) - p\left(t-\frac{1}{a},x\right) \right), \quad x \in \mathbb{R}
	\end{align*}
	with $p(t,0)=\delta(x)$ and $at>1$.
	\begin{remark}
		The condition $at>1$ do not surprise us, since $h(t,\cdot)$ is continuous when $at>1$.
	\end{remark}
	\section{Main results}
	\subsection{Non-local equations}\hfill\\
	In the last section we have seen how in \eqref{eq:Phillips}, from the definition of time-changed process, we can associate to the Variance Gamma process a non-local operator. However, it is not the only possible definition. From \eqref{charVG}, we see that
	\begin{align*}
	e^{-at \ln\left(1+\frac{\xi^2}{b} \right)}=\left(1+\frac{\xi^2}{b}\right)^{-at}=\left(1-i\frac{\xi}{\sqrt{b}}\right)^{-at}\left(1+i\frac{\xi}{\sqrt{b}}\right)^{-at},
	\end{align*}
	then we have an other definition for our process
	\begin{align}
	\label{defVG:difGamma}
	X=G - L
	\end{align}
	where $G$ and $L$ are two independent Gamma subordinators both with parameters $a$ and $\sqrt{b}$. From a financial point of view this representation has the meaning of the difference between independent 'gains' and 'losses' (see \cite{madan1990variance}). This leads us to investigate the connection between the Variance Gamma process and the non-local operators of the Gamma subordinator.
	
	Let us introduce the following generalized Weyl derivatives, for $x \in \mathbb{R}$,
	\begin{align}
	\label{D+}
	\mathcal{D}_{a,b}^+ u(x)&:= \frac{\partial}{\partial x} \int_{-\infty}^{x}  u(s) \overline{\Pi}_{a,b} (x-s) ds,\\
	\label{D-}
	\mathcal{D}_{a,b}^- u(x)&:=-\frac{\partial}{\partial x} \int_{x}^{\infty}  u(s) \overline{\Pi}_{a,b} (s-x) ds,
	\end{align}
	respectively defined for function $u$ such that
	\begin{align*}
	u(s) \overline{\Pi}_{a,b}(x-s) \in L^1(-\infty, x) \quad \text{and} \quad u(s) \overline{\Pi}_{a,b}(s-x) \in L^1(x,\infty), \quad \forall x \in \mathbb{R}.
	\end{align*}
	\begin{remark}
		From \cite[(37)]{colantoni2021inverse}, we know that $\overline{\Pi}_{a,b}(x)=a E_1(bx)$, where $E_1$ is the exponential integral
		\begin{align}
		\label{expint}
		E_1(x):=\int_x^\infty \frac{e^{-z}}{z} dz. 
		\end{align}
		This allows us to write \eqref{D+} and \eqref{D-} in a compact form.
	\end{remark}
	From a simple change of variables, our definitions \eqref{D+} and \eqref{D-} concur with \cite[Definition 2.8]{toaldo} and we achieve 
	\begin{align*}
	\mathcal{D}_{a,b}^+ u(x)&:=\int_{0}^{\infty} \frac{\partial}{\partial x} u(x-s) \overline{\Pi}_{a,b} (s) ds,\\
	\mathcal{D}_{a,b}^- u(x)&:=-\int_{0}^{\infty} \frac{\partial}{\partial x} u(x+s) \overline{\Pi}_{a,b} (s) ds.
	\end{align*}
	The importance of the operators $\mathcal{D}_{a,b}^+$ and $\mathcal{D}_{a,b}^-$ is due to the fact that we can easily compute the characteristic function, essential to the study of L$\acute{e}$vy processes. Indeed, by recalling \eqref{lapTailLm}, we have
	\begin{align}
	\label{charD+}
	\int_{-\infty}^\infty e^{i x \xi } \mathcal{D}_{a,b}^+ u(x) dx&=(-i\xi) \int_{-\infty}^\infty e^{i x \xi } \int_{0}^{\infty}  u(x-s) \overline{\Pi}_{a,b} (s) ds \, dx \notag\\
	&=(-i\xi) \hat{u}(\xi) \int_0^\infty  e^{i s \xi } \overline{\Pi}_{a,b} (s) ds  \notag \\
	&=a  \ln \left(1 - \frac{i\xi}{b} \right) \hat{u}(\xi).
	\end{align}
	Similarly we obtain
	\begin{align}
	\label{charD-}
	\int_{-\infty}^\infty e^{i x \xi } \mathcal{D}_{a,b}^- u(x) dx&=a  \ln \left(1 + \frac{i\xi}{b} \right) \hat{u}(\xi).
	\end{align}
	These results coincide with \cite[Lemma 2.9]{toaldo}.
	We now analyze the connection between these operators and the Variance Gamma process by exploiting \eqref{defVG:difGamma}.
	\begin{theorem}
		\label{thm:non-local}
		Let $p(t,x)$ be the probability density function of the Variance Gamma process $X$. Then we have
		\begin{align*}
		\frac{\partial}{\partial t} p(t,x) =-\left(\mathcal{D}_{a,\sqrt{b}}^+p(t,x) +  \mathcal{D}_{a,\sqrt{b}}^-p(t,x) \right), \quad t>0, x \in \mathbb{R}
		\end{align*}
		with $p(0,x)=\delta(x)$.
	\end{theorem}
	\begin{proof}
		The initial value can be easily checked, since the Fourier transform of the $\delta$ distribution is $1$.  The characteristic function of the left-hand side, with respect to $x$, is given by
		\begin{align*}
		\int_{-\infty}^\infty e^{i x \xi }\frac{\partial}{\partial t} p(t,x)dx= \frac{\partial}{\partial t} e^{-a t \ln \left(1 + \frac{\xi^2}{b} \right)}= - a \ln \left(1 + \frac{\xi^2}{b} \right) e^{-a t \ln \left(1 + \frac{\xi^2}{b} \right)}.
		\end{align*}
		For the right-hand side, using \eqref{charD+} and \eqref{charD-}, we have
		\begin{align*}
		&-\int_{-\infty}^\infty e^{i x \xi }\left(\mathcal{D}_{a,\sqrt{b}}^+p(t,x) +  \mathcal{D}_{a,\sqrt{b}}^-p(t,x) \right)dx\\
		&= -\left(a  \ln \left(1 - \frac{i\xi}{\sqrt{b}} \right)e^{-a t \ln \left(1 + \frac{\xi^2}{b} \right)} + a  \ln \left(1 + \frac{i\xi}{\sqrt{b}} \right)e^{-a t \ln \left(1 + \frac{\xi^2}{b} \right)}\right)\\
		&=-a \ln \left(1 + \frac{\xi^2}{b} \right) e^{-a t \ln \left(1 + \frac{\xi^2}{b} \right)},
		\end{align*}
		then our claim holds true.
	\end{proof}
	\begin{remark}
		The fact that, in the last theorem, there is a sum of non-local derivatives is not new in the theory of probability and non-local operators. For example, for the Brownian motion time-changed with an independent stable subordinator, we know that the one dimensional Riesz derivative can be written as sum of Marchaud derivatives (see \cite[page 12]{ferrari2018weyl}).
	\end{remark}
	In Theorem \ref{thm:non-local} we took advantage of the possible definition of the process $X$ as difference of two independent Gamma subordinators. We observe that this information on the process bring us to the next result for the operators.
	\begin{corollary}
		Let $u$ be a function such that \eqref{Phillips}, \eqref{D+} and \eqref{D-} are well defined, then the following equivalence is satisfied
		\begin{align*}
		-\Phi(-\Delta) u(x)=-\left(\mathcal{D}_{a,\sqrt{b}}^+ u(x) +  \mathcal{D}_{a,\sqrt{b}}^-u(x \right), \quad x \in \mathbb{R}.
		\end{align*}
	\end{corollary}
	\begin{proof}
		We easily provide this result by using the characteristic function. On the left-hand side we have, from \eqref{charPhillips},
		\begin{align*}
		\int_{-\infty}^\infty e^{i \xi x} \left(-\Phi(-\Delta) u(x) \right) dx=-\left({a \ln\left(1+\frac{\xi^2}{b} \right)} \right) \hat{u}(\xi).
		\end{align*}
		On the right-hand side, from \eqref{charD+} and \eqref{charD-}, we have
		\begin{align*}
		&-	\int_{-\infty}^\infty e^{i x \xi }\left(\mathcal{D}_{a,\sqrt{b}}^+u(x) +  \mathcal{D}_{a,\sqrt{b}}^-u(x) \right)dx\\
		&= -\left(a  \ln \left(1 - \frac{i\xi}{\sqrt{b}} \right)\hat{u}(\xi) + a  \ln \left(1 + \frac{i\xi}{\sqrt{b}} \right)\hat{u}(\xi)\right)\\
		&=-a \ln \left(1 + \frac{\xi^2}{b} \right) \hat{u}(\xi),
		\end{align*}
		that concludes the proof.
	\end{proof}
	\textbf{Adding the drift.} 
	We now focus on the Variance Gamma process with drift and we provide that the non-local equation is still true. Let $B^\theta:=\{B_t^\theta, t \geq 0\}$ be the drifted Brownian motion on $\mathbb{R}$ starting from zero, independent from $H$, and $g^\theta(t, x) = e^{-(x-\theta t)^2/ 4t}/\sqrt{4 \pi t}$, for the drift $\theta \in \mathbb{R}$, be its probability density function. The drifted Variance Gamma process $X^\theta:=\{X_t^\theta, t \geq 0\}$ is $X^\theta=B^\theta_H := B^\theta \circ H$. Its characteristic function turns out to be
	\begin{align*}
	\mathbf{E}_0[e^{i \xi X_t^\theta}]=\left(1-i\xi \frac{\theta}{b}+\frac{\xi^2}{b}\right)^{-at}.
	\end{align*}
	We see that, again, the process can be written as difference of two independent Gamma subordinators, as suggested by \cite[(8)]{madan1998variance}, indeed the following holds
	\begin{align*}
	\left(1-i\xi \frac{\theta}{b}+\frac{\xi^2}{b}\right)^{-at}=\left(1-\frac{i\xi}{\sqrt{\frac{\theta^2}{4} + b} -\frac{\theta}{2}}\right)^{-at}\left(1+\frac{i\xi}{\sqrt{\frac{\theta^2}{4} + b} +\frac{\theta}{2}}\right)^{-at},
	\end{align*}
	thus we have $X^\theta=G^\theta-L^\theta$, where $G^\theta$ and $L^\theta$ are two independent Gamma subordinators, with parameters $a$ and $\sqrt{\frac{\theta^2}{4} + b} -\frac{\theta}{2}$ the first one and with parameters $a$ and $\sqrt{\frac{\theta^2}{4} + b} +\frac{\theta}{2}$ the second one. As well as the Variance Gamma process, we show the next result.
	\begin{corollary}
		Let $p^\theta(t,x)$ be the probability density function of the drifted Variance Gamma process $X^\theta$. Then we have
		\begin{align*}
		\frac{\partial}{\partial t} p^\theta(t,x) =-\left(\mathcal{D}_{a,\sqrt{\frac{\theta^2}{4} + b} -\frac{\theta}{2}}^+p^\theta(t,x) +  \mathcal{D}_{a,\sqrt{\frac{\theta^2}{4} + b} +\frac{\theta}{2}}^-p^\theta(t,x) \right), \quad t>0, x \in \mathbb{R}
		\end{align*}
		with $p^\theta(0,x)=\delta(x)$.
	\end{corollary}
	\begin{proof}
		We observe that $p^\theta(t,x)$ can be written as \eqref{ghVG}, where $g^\theta$ replaces $g$. By using that 
		\begin{align*}
		\left(1-i\xi \frac{\theta}{b}+\frac{\xi^2}{b}\right)=\left(1-\frac{i\xi}{\sqrt{\frac{\theta^2}{4} + b} -\frac{\theta}{2}}\right)\left(1+\frac{i\xi}{\sqrt{\frac{\theta^2}{4} + b} +\frac{\theta}{2}}\right),
		\end{align*}
		the proof is analogous to the one of Theorem \ref{thm:non-local}.
	\end{proof}
	\begin{remark}
		Trivially, if $\theta=0$, the last differential equation coincides with the one of Theorem \ref{thm:non-local} as expected.
	\end{remark}
	\subsection{Variance Gamma and special functions}\hfill\\
	In this section we examine how special functions bring us to a new equation for the Variance Gamma process. From \eqref{densityVG}, we have seen that the modified Bessel function $K_\nu(x)$ appears in the density and this special function solves (\cite[formula 9.6.1]{abramowitz1964handbook})
	\begin{align*}
	x^2 \frac{\partial^2}{\partial x^2} u(x) + x \frac{\partial}{\partial x} u(x)- (x^2+\nu^2) u(x)=0,
	\end{align*}
	where, in our case, $\nu$ is a time function. An other interesting fact on the function $K_\nu(x)$ is that it can be written in terms of the Kummer's  (confluent hypergeometric) function usually denoted by $U$, then it is connected to Kummer's equation. This moves our thinking about differential space equations for the Variance Gamma process.
	\begin{theorem}
		The following differential equation is satisfied by the density of the Variance Gamma process $X$:
		\begin{align}
		\label{eq2}
		x \frac{\partial^2}{\partial x^2} p(t,x) -(2at-2) \frac{\partial}{\partial x} p(t,x) - b x p(t,x)=0, \quad t>0,\, x \in \mathbb{R}
		\end{align}
		with $p(0,x)=\delta(x)$.
	\end{theorem}
	\begin{proof}
		The initial value can be easily checked, since the Fourier transform of the $\delta$ is $1$.  The characteristic function, with respect to $x$, of \eqref{eq2} is 
		\begingroup
		\allowdisplaybreaks
		\begin{align*}
		&\int_{-\infty}^\infty e^{i x \xi } \left(x \frac{\partial^2}{\partial x^2} p(t,x) -(2at-2) \frac{\partial}{\partial x} p(t,x) - b x p(t,x)\right) \,dx=\\
		&=-i\frac{\partial}{\partial \xi} \int_{-\infty}^\infty e^{i x \xi }  \frac{\partial^2}{\partial x^2} p(t,x)\,dx - (2at-2)(-i\xi) \int_{-\infty}^\infty e^{i x \xi }  p(t,x)\,dx -b(-i) \frac{\partial}{\partial \xi} \int_{-\infty}^\infty e^{i x \xi }  p(t,x)\,dx\\
		&=-i \frac{\partial}{\partial \xi} (-\xi^2) e^{-a t \ln \left(1 + \frac{\xi^2}{b} \right)} + (2at-2)(i\xi) e^{-a t \ln \left(1 + \frac{\xi^2}{b} \right)}-ib2at \frac{\xi}{b} \left(1+\frac{\xi^2}{b}\right)^{-at-1}\\
		&=2i\xi  \left(1+\frac{\xi^2}{b}\right)^{-at} - 2i\xi^2 at \frac{\xi}{b} \left(1+\frac{\xi^2}{b}\right)^{-at-1} +(2at-2)(i\xi) \left(1+\frac{\xi^2}{b}\right)^{-at} -2iat {\xi} \left(1+\frac{\xi^2}{b}\right)^{-at-1}\\
		&=- 2i\xi^2 at \frac{\xi}{b} \left(1+\frac{\xi^2}{b}\right)^{-at-1} +2iat\xi\left(1+\frac{\xi^2}{b}\right)^{-at}-2iat {\xi} \left(1+\frac{\xi^2}{b}\right)^{-at-1}\\
		&=- 2i\xi^2 at \frac{\xi}{b} \left(1+\frac{\xi^2}{b}\right)^{-at-1}+2iat\xi\left(1+\frac{\xi^2}{b}\right)\left(1+\frac{\xi^2}{b}\right)^{-at-1}-2iat {\xi} \left(1+\frac{\xi^2}{b}\right)^{-at-1}\\
		&=- 2i\xi^2 at \frac{\xi}{b} \left(1+\frac{\xi^2}{b}\right)^{-at-1}+ 2i\xi^2 at \frac{\xi}{b} \left(1+\frac{\xi^2}{b}\right)^{-at-1} +2iat {\xi} \left(1+\frac{\xi^2}{b}\right)^{-at-1}-2iat {\xi} \left(1+\frac{\xi^2}{b}\right)^{-at-1}\\
		&=0,
		\end{align*}
		\endgroup
		as required.
	\end{proof}
	\subsection{Compound Poisson process convergence}\hfill\\
	The convergence of compound Poisson processes to subordinators has been extensively studied. In \cite[Theorem 1]{dovidio2014continuous} the author shows the connection to a difference of $\alpha-$stable subordinators and in \cite[Proposition 3.3]{toaldo} we have the construction for any subordinator. In this section, we exploit these mentioned results to obtain the convergence of a compound Poisson process to the Variance Gamma process.
	
	We consider the i.i.d. random variables $Y_j$, $Y_j \sim Y$ (distributed as $Y$), with probability density function
	\begin{align*}
	\nu_{Y}(y)=\frac{e^{-\sqrt{b} y}}{y} \frac{1}{E_1(\sqrt{b} \gamma)} \mathbf{1}_{y \geq \gamma}, \quad \gamma>0, 
	\end{align*}
	where $E_1$ is defined in \eqref{expint}. Let $\epsilon_j \sim \epsilon$ be i.i.d. (centered)  Rademacher random variables, with law
	\begin{align*}
	\mathbf{P}(\epsilon=+1)=\frac{1}{2}, \quad \mathbf{P}(\epsilon=-1)=\frac{1}{2}.
	\end{align*}
	Now, we define $Y^*=\epsilon Y$ with probability density function
	\begin{align}
	\label{lawY*}
	\nu^*_{Y}=\frac{1}{2} \nu_{Y}(-y) + \frac{1}{2} \nu_{Y}(y). 
	\end{align}
	We are ready to provide the following convergence result.
	\begin{corollary}
		Let $N(t)$, $t \geq 0$, be a homogeneous Poisson process with parameter $1$, independent from the i.i.d random variables $Y^*_j \sim Y^*$, with law \eqref{lawY*}. We have that
		\begin{align}
		\label{convergence}
		\left(\sum_{j=0}^{N(t a E_1(\sqrt{b} \gamma))} Y^*_j\right) \stackrel{law}{\rightarrow} X_{\frac{t}{2}} \quad \text{as }\gamma \to 0.
		\end{align}
	\end{corollary}
	\begin{proof}
		Since we are dealing with a compound Poisson process, then we know that
		\begin{align*}
		\mathbf{E}_0\left[ \exp\left(i \xi \sum_{j=0}^{N(t a E_1(\sqrt{b} \gamma))} Y^*_j \right)\right]&= \exp\left[t a E_1(\sqrt{b} \gamma)(\mathbf{E}_0\left[\exp(i \xi Y^*)\right] -1)\right]\\
		&=\exp\left[t a E_1(\sqrt{b} \gamma)(\mathbf{E}_0\left[\exp(i \xi \epsilon Y)\right] -1)\right]\\
		&=\exp\left[t a E_1(\sqrt{b} \gamma)\left(\frac{1}{2}\mathbf{E}_0\left[\exp(i \xi Y)\right] +\frac{1}{2}\mathbf{E}_0\left[\exp(-i \xi Y)\right] - \left(\frac{1}{2}+\frac{1}{2}\right)\right)\right]\\
		&=\exp\left[t a E_1(\sqrt{b} \gamma)\left(\frac{1}{2}\mathbf{E}_0[\exp(i \xi Y)-1] +\frac{1}{2}\mathbf{E}_0[\exp(-i \xi Y) -1] \right)\right]\\
		&=\exp\left[t\left(\frac{a}{2} \int_\gamma^\infty (e^{i\xi y} - 1)\frac{e^{-\sqrt{b} y}}{y} dy + \frac{a}{2} \int_\gamma^\infty (e^{-i\xi y} - 1)\frac{e^{-\sqrt{b} y}}{y} dy\right)\right]
		\end{align*} 
		If $\gamma \to 0$, we get that
		\begin{align*}
		&\exp\left[t\left(\frac{a}{2} \int_\gamma^\infty (e^{i\xi y} - 1)\frac{e^{-\sqrt{b} y}}{y} dy + \frac{a}{2} \int_\gamma^\infty (e^{-i\xi y} - 1)\frac{e^{-\sqrt{b} y}}{y} dy\right)\right]\\
		 &\to \exp\left[-\frac{t}{2}\left( a  \ln \left(1 - \frac{i\xi}{\sqrt{b}} \right) + a  \ln \left(1 + \frac{i\xi}{\sqrt{b}} \right)\right)\right],
		\end{align*}
		hence the claim holds by using \eqref{defVG:difGamma} and the classic L$\acute{e}$vy's continuity theorem.
	\end{proof}
	\section*{Acknowledgment}
	The author wishes to thank Mirko D'Ovidio for his helpful advice.
	\medskip

\begin{thebibliography}{10}
	
	\bibitem{abramowitz1964handbook}
	Milton Abramowitz and Irene~A. Stegun.
	\newblock {\em Handbook of mathematical functions with formulas, graphs, and
		mathematical tables}, volume~55.
	\newblock US Government printing office, 1964.
	
	\bibitem{beghin2014geometric}
	Luisa Beghin.
	\newblock Geometric stable processes and related fractional differential
	equations.
	\newblock {\em Electron. Commun. Probab.}, 19:1--14, 2014.
	
	\bibitem{bertoin1996levy}
	Jean Bertoin.
	\newblock {\em L{\'e}vy processes}, volume 121.
	\newblock Cambridge university press Cambridge, 1996.
	
	\bibitem{bochner1949diffusion}
	Salomon Bochner.
	\newblock Diffusion equation and stochastic processes.
	\newblock {\em Proc. Natl. Acad. Sci.}, 35(7):368--370, 1949.
	
	\bibitem{colantoni2021inverse}
	Fausto Colantoni and Mirko D'Ovidio.
	\newblock On the inverse gamma subordinator.
	\newblock {\em Stoch. Anal. Appl.}
	\newblock https://doi.org/10.1080/07362994.2022.2108450, 2022.
	
	\bibitem{dovidio2014continuous}
	Mirko D'Ovidio.
	\newblock Continuous random walks and fractional powers of operators.
	\newblock {\em J. Math. Anal. Appl.}, 411(1):362--371, 2014.
	
	\bibitem{ferrari2018weyl}
	Fausto Ferrari.
	\newblock Weyl and {M}archaud derivatives: A forgotten history.
	\newblock {\em Mathematics}, 6(1):6, 2018.
	
	\bibitem{table}
	Izrail~S. Gradshteyn and Iosif~M. Ryzhik.
	\newblock {\em Table of integrals, series, and products}.
	\newblock Elsevier/Academic Press, Amsterdam, seventh edition, 2007.
	\newblock Translated from the Russian, Translation edited and with a preface by
	Alan Jeffrey and Daniel Zwillinger, With one CD-ROM (Windows, Macintosh and
	UNIX).
	
	\bibitem{ken1999levy}
	Sato Ken-Iti.
	\newblock {\em L{\'e}vy processes and infinitely divisible distributions}.
	\newblock Cambridge university press, 1999.
	
	\bibitem{kotz2001laplace}
	Samuel Kotz, Tomasz Kozubowski, and Krzysztof Podg{\'o}rski.
	\newblock {\em The Laplace distribution and generalizations: a revisit with
		applications to communications, economics, engineering, and finance}.
	\newblock Number 183. Springer Science \& Business Media, 2001.
	
	\bibitem{madan1998variance}
	Dilip~B. Madan, Peter~P. Carr, and Eric~C. Chang.
	\newblock The variance gamma process and option pricing.
	\newblock {\em Review of Finance}, 2(1):79--105, 1998.
	
	\bibitem{madan1990variance}
	Dilip~B. Madan and Eugene Seneta.
	\newblock The variance gamma (vg) model for share market returns.
	\newblock {\em J. Bus.}, 63(4):511--524, 1990.
	
	\bibitem{phillips1952generation}
	Ralph~S. Phillips.
	\newblock On the generation of semigroups of linear operators.
	\newblock {\em Pac. J. Math.}, 2(3):343--369, 1952.
	
	\bibitem{toaldo}
	Bruno Toaldo.
	\newblock Convolution-type derivatives, hitting-times of subordinators and
	time-changed {$C_0$}-semigroups.
	\newblock {\em Potential Anal.}, 42(1):115--140, 2015.
	
\end{thebibliography}

	\end{document}